\documentclass[12pt]{amsart}

\usepackage[all]{xy}
\usepackage{xypic}
\usepackage{color}
\usepackage{graphicx}

\newtheorem{theorem}{Theorem}[section]
\newtheorem{lemma}[theorem]{Lemma}
\newtheorem{proposition}[theorem]{Proposition}
\newtheorem{remark}[theorem]{Remark}

\numberwithin{equation}{section}

\begin{document}

\baselineskip=15.5pt

\title{On the logarithmic connections over curves}

\author[I. Biswas]{Indranil Biswas}

\address{School of Mathematics, Tata Institute of Fundamental
Research, Homi Bhabha Road, Bombay 400005, India}

\email{indranil@math.tifr.res.in}

\author[V. Heu]{Viktoria Heu}

\address{IRMA, UMR 7501, 7 rue Ren\'e-Descartes, 67084 Strasbourg Cedex, France}

\email{heu@math.unistra.fr}



\date{}

\begin{abstract}
We study two different actions on the moduli spaces of logarithmic connections 
over smooth complex projective curves. Firstly, we establish a dictionary 
between logarithmic orbifold connections and parabolic logarithmic connections 
over the quotient curve. Secondly, we prove that fixed points on the moduli 
space of connections under the action of finite order line bundles are exactly 
the push-forward of logarithmic connections on a certain unramified Galois 
cover of the base curve. In the coprime case, this action of finite 
order line bundles on the moduli space is cohomologically trivial.
\end{abstract}

\maketitle

\section{Introduction}

There is a well-established dictionary between the orbifold 
vector bundles, \emph{i.e.,} holomorphic vector bundles endowed with an action 
of a finite subgroup $\Gamma$ of automorphisms of a smooth complex projective 
base curve $Y$, and the parabolic vector bundles over the 
quotient curve $X\,=\,Y/\Gamma$. In the first part of this paper, we 
generalize this to a dictionary between the orbifold vector bundles $E 
\longrightarrow Y$ endowed with a $\Gamma$--equivariant logarithmic connection 
and the vector bundles $V\longrightarrow X$ endowed with a parabolic logarithmic 
connection $D'$. The $\Gamma$--equivariance property of $D$ means that the 
action of $\Gamma$ on locally
defined sections preserves the property of being flat with respect to 
$D$. By parabolic logarithmic connection we mean that
\begin{itemize}
\item the parabolic structure 
on $V$ induced by $D'$ coincides with the parabolic structure on $V$ induced by 
the orbifold vector bundle $E$, and

\item the residue splits the quasiparabolic filtration.
\end{itemize} 
We establish the following theorem in Section \ref{pareq}.

\begin{theorem}
Let $Y$ be a compact connected Riemann surface. Let $\Gamma$ be a finite 
subgroup of $\mathrm{Aut}(Y).$ Let $S_Y$ be a finite subset of $Y$ such that 
$\Gamma$ acts freely on $S_Y$. Let $(E,D)$ be a logarithmic connection on $Y$,
with singular set included in $S_Y$, equipped with a lift of the 
tautological action of $\Gamma$. Let $\phi$ be the quotient map 
$Y\,\longrightarrow\, X \,:=\,Y/\Gamma$. Then the push-forward connection 
$\phi_*D$ induces a logarithmic connection $D'$ on the $\Gamma$-invariant 
sub-vector bundle $V\,=\,(\phi_*E)^\Gamma$ of $\phi_*E$. This connection 
preserves the natural quasiparabolic filtration $$V_x=V_x^0\supset V_x^1 \supset 
V_x^2 \subset \ldots \supset V_x^{\ell_x} \supset V_x^{\ell_x+1}= 0$$
for each ramification point $x \in X$ of the map $\phi$, and the parabolic 
weights associated to this filtration coincide with the eigenvalues of the 
endomorphism of $V_x^i / V_x^{i+1}$ induced by the residue of the connection.
Furthermore, the residue of the logarithmic connection $D'$ at any parabolic 
point splits completely the quasiparabolic filtration.
\end{theorem}

In Section \ref{eqpar}, we carry out a reverse construction from parabolic 
connections on $X$ to $\Gamma$--equivariant connections on a certain cover 
$Y\,\longrightarrow\,X$ with $X\,=\,Y/\Gamma$.

This work has been motivated by the fact that another result about certain 
group actions on the moduli space of semi-stable vector bundles can be 
generalized to the corresponding action on the moduli space of logarithmic 
connections. To explain this, consider the action by tensor product of the
holomorphic line bundle $L$ of order $n$ on the moduli space 
$\mathcal{M}(r,\xi)$ of 
semi-stable vector bundles of rank $r \,\in\, n\mathbb{N}$ with determinant 
line bundle $\xi$ over a Riemann surface $X$ of genus at least $2$.
The fixed points of this action 
are exactly the 
vector bundles that arise as the push-forward $(\phi_L)_*\eta$ of a line bundle 
$\eta$ over $Y_L$ such that $\mathrm{det}((\phi_L)_*\eta)\,=\,\xi$, where 
$\phi_L\,:\, Y_L\,\longrightarrow\, X$ is the unramified $n$-cover defined by 
$L$ \cite{BP} (see Section \ref{se3.1} for the details). For vector bundles 
endowed with logarithmic
connections, we prove the following theorem in Section \ref{se3.1}.

\begin{theorem}\label{thm-i2}
Let $L$ be a line bundle of order $n$ over a compact connected Riemann surface 
$X$, and let $\phi_L : Y_L \longrightarrow X$ be the unramified cover of degree 
$n$ given by $L$. Let $r$ be a multiple of $n$, and let $D_L$ be the unique 
(holomorphic) connection on $L$ that induces the trivial 
connection on $L^{\otimes n}$. Let $(E,D)$ be a logarithmic connection on $X$.
There is an isomorphism of connections $(E\, ,D)\,\simeq\,(E \otimes L\, , 
D\otimes \mathrm{Id}_L+\mathrm{Id}_E \otimes D_L)$ if and only if $(E,D)$ 
coincides with the push-forward $({\phi_L}_*V, {\phi_L}_*D_V)$ of a logarithmic 
connection $(V,D_V)$ of rank $r/n$ on $Y_L.$
\end{theorem}

Theorem \ref{thm-i2} classifies the fixed points under the action by tensor 
product with finite order line bundles on certain moduli spaces of (holomorphic 
or) logarithmic connections. Moreover, this action is cohomologically trivial 
in a sense that will be explained below. 

Let $X$ be a compact connected Riemann surface of genus at least two. Fix a 
line bundle $\xi$ of degree $d$ over $X$ and a logarithmic connection $D_\xi$ 
on $\xi$ singular exactly over $x_0\,\in\, X$. Let $r\,\geq\, 2$ be an integer 
coprime to $d$. Denote by $\mathcal{M}(r,D_\xi)$ the (smooth) moduli space of 
logarithmic connections of rank $r$ over $X$ that are singular exactly over 
$x_0$ with residue $-\frac{d}{r}\cdot \mathrm{Id}$ such that the induced 
connection 
on the determinant line bundle coincides with $(\xi, D_\xi)$. Let $L$ be a 
line bundle of order $n \,|\,r$ over $X$. 
In Section \ref{se3.2}, we prove the following theorem.

\begin{theorem}
The homomorphism 
$$\begin{array}{rrcl}\Phi^*: & \mathrm{H}^*(\mathcal{M}(r,D_\xi), \mathbb{Q})& \longrightarrow & \mathrm{H}^*(\mathcal{M}(r,D_\xi), \mathbb{Q})\end{array}$$
induced by the automorphism $$\begin{array}{rccc}\Phi : & \mathcal{M}(r,D_\xi) & \longrightarrow &\mathcal{M}(r,D_\xi)\\
&(E,D)&\longmapsto & (E \otimes L, D\otimes \mathrm{Id}_L+\mathrm{Id}_E \otimes D_L),\end{array}$$
	is the identity map.
\end{theorem}
 
Finally, in Section \ref{se4} we generalize the Atiyah--Weil criterion for the existence of a 
holomorphic connection on a given vector bundle $E \longrightarrow X$ into a 
criterion for the existence of a logarithmic connection on $E$ with prescribed 
central residues.

\section{Equivariant connections and parabolic connections}

\subsection{Parabolic connections from equivariant connections}\label{pareq}

Let $Y$ be a compact connected Riemann surface. The group of all
holomorphic automorphisms of $Y$ will be denoted by $\text{Aut}(Y)$. Let
\begin{equation}\label{e1}
\Gamma\, \subset\, \text{Aut}(Y)
\end{equation}
be a finite subgroup. A $\Gamma$--\textit{equivariant} holomorphic vector
bundle on $Y$ is a holomorphic vector bundle $E\, \longrightarrow\,Y$ equipped 
with a lift of the tautological action of $\Gamma$ on $Y$. More precisely, a 
$\Gamma$--equivariant holomorphic vector bundle on $Y$ is a pair of the 
form $(E\, ,\rho)$, 
where $E\, \longrightarrow\, Y$ is a holomorphic vector bundle, and $\rho$ is an 
action of $\Gamma$ on the total space of $E$ such that for each $\gamma\, \in\, 
\Gamma$, the map
$$
\rho(\gamma)\, :\, E\, \longrightarrow\, E
$$
is a holomorphic isomorphism of the vector bundle $E$ with the pullback
$(\gamma^{-1})^* E$. Clearly, this condition implies that the action $\rho$ is 
holomorphic.

Fix a reduced 
effective divisor $S_Y$ on $Y$ such that
\begin{enumerate}
\item the tautological action of $\Gamma$ on $Y$ preserves $S_Y$, meaning
$\Gamma(S_Y)\, \subset\, S_Y$, and

\item the action of $\Gamma$ on $S_Y$ is free.
\end{enumerate}
The holomorphic cotangent bundle of $Y$ will be denoted by $K_Y$. Since 
$\gamma(S_Y)\, =\, S_Y$ for all $\gamma\, \in\, \Gamma$,
the tautological action of $\Gamma$ on $X$ lifts naturally to the total space
of the holomorphic line bundle $K_Y\otimes {\mathcal O}_Y(S_Y)$. This action on 
$K_Y\otimes {\mathcal O}_Y(S_Y)$ shall be denoted by $\delta$.

A \textit{logarithmic connection} on a holomorphic vector bundle $E\, 
\longrightarrow\, Y$ singular 
over $S_Y$ is a holomorphic differential operator
$$
D\, :\, E\, \longrightarrow\, E\otimes K_Y\otimes {\mathcal O}_Y(S_Y)
$$
satisfying the Leibniz identity, which says that $D(fs) \,=\, f\cdot ds+
s\otimes (df)$, where $f$ is any locally defined holomorphic function on $Y$
and $s$ is any locally defined holomorphic section of $E$.
The singular locus of a logarithmic connection $D$ singular over $S_Y$ is 
contained in $S_Y$.

Assume that $(E\, ,\rho)$ is $\Gamma$--equivariant. A logarithmic connection 
$D$ on $E$ is called $\Gamma$--\textit{equivariant} if the action $\rho$
preserves $D$, meaning
$$
D(\rho(\gamma)(s)) \,=\, (\rho(\gamma)\otimes\delta(\gamma))(D(s))
$$
for all $\gamma\, \in\, \Gamma$ and all locally defined holomorphic section $s$ 
of $E$, where $\delta$ is the action defined earlier (if $s$ is defined over $U\, 
\subset\, Y$, then the section $\rho(\gamma)(s)$ is defined over $\gamma(U)$).

The singular locus of a $\Gamma$--equivariant logarithmic connection
is clearly preserved by the action of $\Gamma$.

Define
\begin{equation}\label{e2}
X\,:=\, Y/\Gamma\, .
\end{equation}
So $X$ is a compact connected Riemann surface. Let
\begin{equation}\label{e3}
\phi\, :\, Y\, \longrightarrow\, X
\end{equation}
be the quotient map. This $\phi$ is a Galois covering with Galois
group $\Gamma$. Let
\begin{equation}\label{e4}
R_{Y}\, \subset\, Y
\end{equation}
be the subset with nontrivial isotropy for the tautological action of $\Gamma$
on $Y$. The map $\phi$ in \eqref{e3} fails to be \'etale precisely over
\begin{equation}\label{e5}
R_X\, :=\, \phi(R_Y)\, \subset\, X\, .
\end{equation}
Define
\begin{equation}\label{e6}
S_X\, :=\, \phi(S_Y) ~\, ~\text{ and } ~\, ~ \Delta\, :=\, S_X\cup R_X\,
\subset\, X\, .
\end{equation}
We have $S_X\bigcap R_X\,=\, \emptyset$, because the action of $\Gamma$ 
on $S_Y$ is free.

Given a $\Gamma$--equivariant holomorphic vector bundle on $Y$, there is a 
canonically associated parabolic vector bundle of same rank over $X$ with 
parabolic structure over $R_X$ \cite[Section 2c]{Bi}. Let $E$ be a
$\Gamma$--equivariant vector bundle on $Y$, and let $V_*$ be the
parabolic vector bundle on $X$ associated to $E$. If $V$ is the
holomorphic vector bundle on $X$ underlying the parabolic vector bundle
$V_*$, then
\begin{equation}\label{e7}
V\, =\, (\phi_*E)^\Gamma\, \subset\, \phi_*E\, ,
\end{equation}
where $(\phi_*E)^\Gamma$ is the invariant part for the natural action of
the Galois group $\Gamma$ on $\phi_*E$ (see \cite[p. 310, (2.9)]{Bi}).

For each point $y\, \in\, Y$, let
$$
\Gamma_y\, :=\, \{\gamma\, \in\, \Gamma\, \mid\, \gamma(y)\,=\, y\}\, \subset\, 
\Gamma
$$
be the isotropy subgroup. The order of $\Gamma_y$ depends only on $x 
\,=\, \phi(y)$ and shall be denoted by
\begin{equation}\label{nx}
n_x\, :=\, \#\Gamma_y\, .
\end{equation}

The parabolic structure on $V$ then is given by the decreasing, left 
continuous filtration $\{V_t\}_{t \in \mathbb{R}}$ defined by
 \begin{equation}\label{e7b}
V_t\, =\, (\phi_*
(E\otimes \mathcal{O}_Y(\sum_{x \in {R}_X}\lfloor -t n_x\rfloor 
\phi^{-1}(x)_{\textrm{red}})) )^\Gamma\, ,
\end{equation}
where $\phi^{-1}(x)_{\rm red}$ is the set--theoretic
inverse image of $x$ (the reduced inverse image) (see \cite[p. 310, 
(2.9)]{Bi}); here $\lfloor Z \rfloor$ denotes the integer 
part of $Z$ and the action of $\Gamma$ on 
$\mathcal{O}_Y(k\cdot\phi^{-1}(x)_{\textrm{red}})$ is the 
natural one (note that $\phi^{-1}(x)_{\textrm{red}}$ is invariant under 
$\Gamma$). We say that $\alpha \,\in\, [0,1)$ is a jump for this filtration if 
$V_{\alpha+\varepsilon}\neq V_{\alpha}$ for all $\varepsilon >0$. 
There is a finite number of such jumps, which we shall denote by
$$0\,\leq\, \alpha_0\,<\,\alpha_1\,<\, \cdots\, <\, \alpha_\ell\,<\,1\, .$$
These will be the parabolic weights associated to the quasiparabolic filtration
$$V\,=\,V_{\alpha_0}\,\supset\, V_{\alpha_1}\,\supset\, \cdots \,\supset\, 
V_{\alpha_\ell} \,\supset\, V_1 \,=\,V \otimes \mathcal{O}(-R_X)\, .$$
By taking the quotient by $V\otimes \mathcal{O}_X(-R_X)$, we can associate,
for each point $x \in R_X$, a filtration (with corresponding weights) by 
sub-vector spaces of the fiber $V_x$:
$$V_x\,=\,V_x^0\,\supset\, V^1_x\,\supset\,\cdots\, \supset\, V^{\ell_x}_x 
\,\supset\, V_x^{\ell_x+1}\,=\,0\, .$$

Take any $y\, \in\, R_Y$, and denote $\phi(y)$ by $x$. Let $\chi_0$ be the 
character of $\Gamma_y$ given by the tautological action of $\Gamma_y$ on 
$T_yY$. Consider the action of $\Gamma_y$ on $E_y$ given by $\rho$. From \eqref{e7b} it follows that $$E_y=\bigoplus_{i=0}^{\ell_x} W_i,$$ where 
 $W_i$ is the isotypical component of $E_y$ for 
the character $\chi^{n_x\alpha_i}_0$, \emph{i.e.}
$$
W_i\, =\, \{v\, \in\, E_y\, \mid\, \rho(\gamma)(v)\,=\, 
\chi^{n_x\alpha_i}_0(\gamma)\cdot v~\, \forall~ \gamma\, \in\, \Gamma_y\}\, .
$$
{}Moreover, there is a canonical isomorphism
\begin{equation}\label{ci}
V^i_x/V^{i+1}_x\,\stackrel{\sim}{\longrightarrow}\, W_i\otimes 
(K_Y\vert_y)^{\otimes n_x\alpha_i}\, .
\end{equation}
for each $i\,\in\,[0,\ell_x]$ (the fiber $\mathcal{O}_Y(-\{y\})|_y$ is 
canonically identified to $K_Y|_y$ by the Poincar\'e adjunction formula).

The holomorphic cotangent bundle of $X$ will be denoted by $K_X$.

\begin{lemma}\label{lem1}
Let $(E\, ,\rho \, ,D)$ be a $\Gamma$--equivariant logarithmic connection on 
$Y$. Let $V$ be the holomorphic vector bundle on $X$ underlying the parabolic 
vector bundle associated to $E$. Then $D$ induces a logarithmic connection
$$
V\, \stackrel{D'}{\longrightarrow}\, V\otimes K_X\otimes {\mathcal O}_X(\Delta)
$$
on $V$, where $\Delta$ is defined in \eqref{e6}.
\end{lemma}

\begin{proof}
We first observe that $K_Y\otimes {\mathcal O}_Y(S_Y)$ is a subsheaf of
$\phi^*(K_X\otimes {\mathcal O}_X(\Delta))$ using the differential
$d\phi\, :\, \phi^*K_X\,\longrightarrow\, K_Y$ of $\phi$. Indeed, this follows 
from the fact that for any point $x\,\in\, X$, the differential $d\phi$ 
identifies $\phi^*(K_X\otimes{\mathcal O}_X(x))$ with $K_Y\otimes {\mathcal 
O}_Y(\phi^{-1}(x)_{\rm red})$, where $\phi^{-1}(x)_{\rm red}$ as before is the 
set--theoretic inverse image of $x$. Consider the composition
$$
E\,\stackrel{D}{\longrightarrow}\,E\otimes K_Y\otimes {\mathcal O}_Y(S_Y)
\, \hookrightarrow\, E\otimes \phi^*(K_X\otimes {\mathcal O}_X(\Delta))\, .
$$
Taking direct image of both sides, and using the projection formula, we have
$$
\phi_*E \, \longrightarrow\, \phi_*(E\otimes \phi^*(K_X\otimes {\mathcal 
O}_X(\Delta)))\,=\, (\phi_*E) \otimes K_X\otimes {\mathcal O}_X(\Delta)\, .
$$
It is straight--forward to check that this homomorphism satisfies the
Leibniz identity, hence it is a logarithmic connection on $\phi_*E$. Let 
$\widetilde D$ denote this logarithmic connection on $\phi_*E$.

Since the action $\rho$ of $\Gamma$ on $E$ preserves the logarithmic connection 
$D$, from the above construction of $\widetilde D$ it follows immediately that 
$\widetilde D$ commutes with the natural action of the Galois group $\Gamma$ on 
$\phi_*E$. Therefore, $\widetilde D$ induces a logarithmic connection $D'$ on the 
invariant part $(\phi_*E)^\Gamma$. Now the lemma follows from the isomorphism in 
\eqref{e7}.
\end{proof}

Since the tautological action of $\Gamma$ on $Y$ is faithful, and $\Gamma$ is a 
finite group, it follows that the isotropy subgroup $\Gamma_y$ is a finite 
cyclic group. In fact, the action of $\Gamma_y$ on the holomorphic tangent line 
$T_y Y$ is faithful. We have $\Gamma_y\,=\,\{e\}$ for all $y\, \in\, 
Y\setminus R_Y$. In particular, $\Gamma_y\,=\, \{e\}$ for all $y\, \in\, S_Y$.
Let $\Gamma^\vee_y$ denote the group of characters of $\Gamma_y$.

As before, let $(E\, ,\rho\, ,D)$ be a $\Gamma$--equivariant logarithmic 
connection on $Y$. For any $y\, \in\, Y$, let
\begin{equation}\label{e8}
E_y\,=\, \bigoplus_{\chi\in \Gamma^\vee_y} E^\chi_y
\end{equation}
be the isotypical decomposition for the action of $\Gamma_y$ on the fiber $E_y$
given by $\rho$. Here $E_y^\chi \,=\, \{v \in E_y 
~|~\rho(\gamma)(v)\,=\,\chi(\gamma)\cdot 
v$ for all $\gamma \,\in\, \Gamma_y\}$. It should be clarified that $E^\chi_y$ 
may be zero for some character $\chi$. Let
$$
\text{Res}(D\, ,y)\, \in\, \text{End}(E_y)\, =\, E_y\otimes E^\vee_y
$$
be the residue of $D$ at $y\, \in\, S_Y$.

Let $V_*$ be the parabolic vector bundle, with parabolic structure over
$R_X$, associated to $E$. As before, the holomorphic vector bundle underlying 
$V_*$ will be denoted by $V$. Take any point $x\, \in\, R_X$. Let
\begin{equation}\label{e9}
0\, =\, V^{\ell_x+1}_x\, \subset\, V^{\ell_x}_x \, \subset\, \cdots \, \subset\,
V^1_x \, \subset\, V^0_x\,=\, V_x
\end{equation}
be the quasiparabolic filtration for $V_*$ at the point $x$. Take any point
$y \,\in\,\phi^{-1}(x)$, and let $n_x$ be as in \eqref{nx}. Let $\chi_0$ be the 
character of $\Gamma_y$ given by the action of 
$\Gamma_y$ on the tangent line $T_y Y$. If $m_i/n_x$, $0\,\leq\, i\, \leq\, 
\ell_x$, is the parabolic weight of the subspace $V^i_x\, \subset\, V_x$ in 
\eqref{e9}, then $V^i_x/V^{i+1}_x$ is canonically identified with
$$
E^{\chi^{m_i}_0}_y\otimes(K_Y\vert_y)^{\otimes m_i}\, ,
$$
where $E^{\chi^{m_i}_0}_y\, \subset\, E_y$ is the direct summand in \eqref{e8} 
(see \eqref{ci}). This implies that
$${\rm End}(V^i_x/V^{i+1}_x)\,=\, {\rm End}(E^{\chi^{m_i}_0}_y)\, .$$

\begin{proposition}\label{prop}
Take a point $x\, \in\, R_X$. Let ${\rm Res}(D'\, ,x)$ be the residue, at $x$, 
of the logarithmic connection $D'$ on $V$ given by $D$ (see Lemma \ref{lem1}). 
Then there is a complete splitting of the filtration in \eqref{e9} satisfying 
the following conditions:
\begin{enumerate}
\item each direct summand in the complete splitting is preserved by
${\rm Res}(D'\, ,x)$, in particular, ${\rm Res}(D'\, ,x)$ produces an
endomorphism of each sub-quotient $V^i_x/V^{i+1}_x$, $0\, \leq\, i\, \leq\,
\ell_x$, in \eqref{e9}, and

\item the restriction of ${\rm Res}(D'\, ,x)$ to the direct summand of
$V_x$ corresponding to $V^i_x/V^{i+1}_x$ is $\frac{m_i}{n_x}{\rm Id}$
(recall that $m_i/n_x$ is the parabolic weight of $V^i_x/V^{i+1}_x$, where
$n_x\,=\, \# \Gamma_x$).
\end{enumerate}
\end{proposition}

\begin{proof}
Take any point $y\, \in\, \phi^{-1}(x)$. Since $x\, \in\, R_X$, the
connection $D$ is nonsingular in a neighborhood of $y$. Consider the 
decomposition 
$$
E_y\,=\, \bigoplus_{\chi\in \Gamma^\vee_y} E^\chi_y
$$
in \eqref{e8}. Since $D$ is a flat connection on $E$ over $Y\setminus S_Y$, 
taking parallel translations for $D$, this decomposition produces a holomorphic
decomposition
\begin{equation}\label{e11}
E\vert_U\,=\, \bigoplus_{\chi\in \Gamma^\vee_y} F^\chi
\end{equation}
over a sufficiently small contractible analytic neighborhood $U\, \subset\, Y$ of 
the point $y$ such that
\begin{enumerate}
\item $\gamma(U)\, =\, U$ for all $\gamma \in \Gamma_y$, and

\item the natural map $U/\Gamma_y\, \longrightarrow\, X$ defined by $\phi$ is 
injective.
\end{enumerate}
Such a neighborhood can be found because in some suitable coordinate charts 
near $y\,\in \, Y$ and $x\,\in\, X$, the restriction of $\phi$ is 
given by $z \,\longmapsto\, z^{n_x}$.

The above defined $F^\chi$ is the unique flat subbundle of $E\vert_U$ such 
that $F^\chi_y\,=\,E^\chi_y$; it is easy to check that $F^\chi$ is preserved 
under the action of $\Gamma_y$ on $E\vert_U$. It should be clarified that the 
decomposition of $E\vert_U$ in \eqref{e11} does not depend on anything other 
than the connection $D$.

For any $\gamma\, \in\, \Gamma$ and $\gamma_1\, \in\, \Gamma_y$, the subbundle 
$\rho(\gamma\gamma_1)(F^\chi)$ of $E\vert_{\gamma\gamma_1(U)}$ coincides with 
the subbundle $\rho(\gamma)(F^\chi)$, where $\rho$ is the action of $\Gamma$ on $E$ 
(note that $\gamma\gamma_1(U)\,=\, \gamma(U)$). The decomposition in 
\eqref{e11} produces a holomorphic decomposition
\begin{equation}\label{e13}
E\vert_{\phi^{-1}(\phi(U))}\,=\, E\vert_{\Gamma (U)}\,=\,
\bigoplus_{\chi\in \Gamma^*_y} \Gamma (F^\chi)
\end{equation}
of the vector bundle. The action of $\Gamma$ on $E$ clearly preserves $\Gamma 
(F^\chi)$.
It is straight-forward to check that the decomposition in \eqref{e13} is 
independent of the choice of the point $y\, \in\, \phi^{-1}(x)$.
In view of \eqref{e7}, the decomposition in \eqref{e13}
produces a holomorphic decomposition
\begin{equation}\label{e12}
V_*\vert_{\phi(U)} \,=\, \bigoplus_{\chi\in \Gamma^*_y} W^\chi_*
\end{equation}
of the parabolic vector bundle over $\phi(U)$; here $W^\chi_*$ is the
parabolic vector bundle on $\phi(U)$ corresponding to the 
$\Gamma$--equivariant vector bundle $\Gamma (F^\chi)\, \longrightarrow\,
\phi^{-1}(\phi(U))$.

Let $W^\chi_x$ be the fiber, over $x$, of the holomorphic vector bundle
underlying the parabolic vector bundle $W^\chi_*$ in \eqref{e12}. We recall 
that the group $\Gamma_y$ acts on the fiber $F^\chi\vert_y\,=\, E^\chi_y$ as
scalar multiplications through the character $\chi$. This implies that
the quasiparabolic filtration of the fiber $W^\chi_x$ is the trivial filtration
$$
0\, \subset\, W^\chi_x\, .
$$
In other words, $W^\chi_*$ has at most one parabolic weight at $x$. Recall that
$W^\chi_*$ can be of rank zero for some $\chi$; if the rank of $W^\chi_*$ is 
positive, then there is exactly one parabolic weight of $W^\chi_*$ at $x$. If 
$W^\chi_*$ is of positive rank, then the parabolic weight of $W^\chi_*$ at $x$ 
is
\begin{equation}\label{pw}
u/n_x\, ,
\end{equation}
where $n_x$ is defined in \eqref{nx}, and $u\, \in\, [0\, , 
n_x-1]$ is determined by the identity $\chi^u_0\, =\, \chi$ (as before,
$\chi_0$ is the character given by the action of $\Gamma_x$ on $T_yY$).

The decomposition of $V_*\vert_{\phi(U)}$ in \eqref{e12} yields a complete 
splitting
\begin{equation}\label{e14}
V_x \,=\, \bigoplus_{\chi\in \Gamma^*_y} W^\chi_x
\end{equation}
of the filtration in \eqref{e9}.

The decomposition in \eqref{e13} is preserved by the connection $D$. Therefore, 
each $W^\chi_*$ in \eqref{e12} is preserved by the connection 
$D'\vert_{\phi(U)}$ on $V\vert_{\phi(U)}$ (see Lemma \ref{lem1} for $D'$). 
Consequently, the residue ${\rm Res}(D'\, ,x)$ preserves the subspace 
$W^\chi_x\, \subset\, V_x$ for each $\chi$.

To prove the final statement of the proposition, we consider a local model of
$\phi$. Let
\begin{equation}\label{d}
{\mathbb D}\, :=\, \{z\, \in\, \mathbb C\, \mid\, |z| \, \leq\, 1\}\,
\subset\, \mathbb C
\end{equation}
be the open unit disk, and let
\begin{equation}\label{d2}
f_n\,:\,{\mathbb D}\,\longrightarrow\, {\mathbb D}\, , ~ \, n\, \geq\, 2,
\end{equation}
be the holomorphic map defined by $z\, \longmapsto\, z^n$. 
So the Galois group of $f_n$ is $\text{Gal}(f_n)\,=\, {\mathbb Z}/n{\mathbb 
Z}$. Consider
the trivial holomorphic line bundle ${\mathcal O}_{\mathbb D}$ over $\mathbb D$. 
Fix an integer $m\, \in\, [0\, ,n-1]$ and consider the action of ${\mathbb 
Z}/n{\mathbb Z}$ on ${\mathcal O}_{\mathbb D}$ where the generator $1\, \in\, 
{\mathbb Z}/n{\mathbb Z}$ acts as $(z\, ,\lambda)\, 
\longmapsto\, (z\cdot \exp(2\pi\sqrt{-1}/n)\, 
,\lambda\cdot \exp(-2\pi\sqrt{-1}\cdot m/n))$, 
for all $z\, \in\, \mathbb D$ and $\lambda\,\in \, \mathbb C$. Let
$$
((f_n)_*{\mathcal O}_{\mathbb D})^{\text{Gal}(f_n)}\, \longrightarrow\,
f_n({\mathbb D})
$$
be the holomorphic line bundle given by the invariant part of the direct image. 
The holomorphic connection $D$ on ${\mathcal O}_{\mathbb D}$ given by the de 
Rham differential $\beta\, \longmapsto\, d\beta$, being invariant under the 
action of $\text{Gal}(f_n)$, produces a logarithmic connection
on $((f_n)_*{\mathcal O}_{\mathbb D})^{\text{Gal}(f_n)}$ singular over $0\, 
\in\, f_n({\mathbb D})$. It is straight-forward to check that the residue of 
this logarithmic connection at $0$ is $m/n$: If $e$ is the constant function 
$1$ on ${\mathbb D}$, then $(e, z\cdot e, \ldots, z^{n-1}\cdot e)$ is a basis 
of the ${\mathcal O}_{\mathbb D}$--module $(f_n)_*\mathcal{O}_{\mathbb D}$. We 
have $$(f_n)_*D(z^ke)\,=\, kz^k\cdot e 
\otimes \mathrm{d}z/z \,=\, k/n\cdot z^k\cdot e \otimes \mathrm{d}x/x$$ with 
$x=z^n$, for $0\leq k\leq n-1$. Now $\Gamma$ acts on the sections of $(f_n)_*(\mathcal{O}_\mathbb{D})$ as $(z,z^k\cdot e) \longmapsto (z\cdot \exp(2\pi\sqrt{-1}/n)\, 
, \exp(2\pi\sqrt{-1}\cdot (k-m)/n)z^k\cdot e)$. The invariant section of 
$(f_n)_*\mathcal{O}_{\mathbb D}$ thus is $z^k\cdot e$ with $k=m$, and the 
residue of the resulting connection on $((f_n)_*\mathcal{O}_\mathbb{D})^\Gamma$ 
at $0$ is $m/n$.

Using this local model, from (\ref{e7b}) it follows immediately
that the restriction of the residue ${\rm Res}(D'\, ,x)$ to the direct summand
$W^\chi_x$ in \eqref{e14} is $\lambda_\chi\cdot \text{Id}$, where $\lambda_\chi$ 
is the parabolic weight of $W^\chi_*$ at $x$. Recall that there is exactly one
parabolic weight if the rank of $W^\chi_*$ is positive; the parabolic weight 
is described in \eqref{pw}.
\end{proof}

\subsection{Equivariant connections from parabolic connections} \label{eqpar}

Let $X$ be a compact connected Riemann surface. Fix finite subsets
$$
R_X\,\subset\, \Delta\, \subset\, X\, .
$$
Let $V$ be a parabolic vector bundle over $X$ with parabolic structure over 
$R_X$ such that all the parabolic weights are rational numbers. For any
point $x\, \in\, R_X$, let
\begin{equation}\label{f1}
0\, =\, V^{\ell_x+1}_x\, \subset\, V^{\ell_x}_x \, \subset\, \cdots \, \subset\,
V^1_x \, \subset\, V^0_x\,=\, V_x
\end{equation}
be the quasiparabolic filtration at $x$. Let
\begin{equation}\label{f2}
1\, >\,\alpha^{\ell_x}_x \, >\,\cdots \, >\,\alpha^1_x \, >\, \alpha^0_x\,\geq\, 0
\end{equation}
be the corresponding parabolic weights. We have $\alpha^i_x\,\in\, \mathbb Q$ by
our assumption.

Let
$$
D'\, :\, V\, \longrightarrow\, V\otimes K_X\otimes{\mathcal O}_X(\Delta)
$$
be a logarithmic connection on $V$ satisfying the following two conditions:
\begin{enumerate}
\item for each $x\, \in\, R_X$, there is a complete splitting
\begin{equation}\label{f3}
V_x\, =\, \bigoplus_{i=0}^{\ell_x} W^i_x
\end{equation}
of the filtration in \eqref{f1} such that the residue ${\rm Res}(D'\, ,x)$
preserves each subspace $W^i_x$, and

\item the restriction of ${\rm Res}(D'\, ,x)$ to each $W^i_x$ is multiplication 
by the weight $\alpha^i_x$ in \eqref{f2}.
\end{enumerate}

Fix a positive integer $N$ such that $N\cdot \alpha^i_x$ is an integer for
all $x\, \in\, R_X$ and $i\,\in\, [0\, , \ell_x]$.

There is a finite (ramified) Galois covering 
\begin{equation}\label{f4}
\phi\, :\, Y\, \longrightarrow\, X
\end{equation}
such that
\begin{itemize}
\item at each point $y\, \in\, \phi^{-1}(R_X)$, the map $\phi$ is ramified
with ramification index $N-1$ (so $\phi$ is like $z\, \longmapsto\, z^N$
around each point of $\phi^{-1}(R_X)$), and

\item $\phi$ is unramified over each point of $\Delta\setminus R_X$.
\end{itemize}
(See \cite[p. 29, Theorem 1.2.15]{Na} for the existence of such a covering.) 
In the special case where $\text{genus}(X)\,=\, 0\,=\, \# R_X- 1$, the map 
$\phi$ is necessarily ramified
over a point of $X\setminus \Delta$. In all other cases however, the
covering map $\phi$ can be chosen to be unramified over $X\setminus R_X$
\cite[p. 29, Theorem 1.2.15]{Na}. Define
$$
S_X\, :=\, \Delta\setminus R_X.
$$

Let
$$
\Gamma\, :=\, \text{Gal}(\phi)
$$
be the Galois group for the covering $\phi$. There is a $\Gamma$--equivariant 
vector bundle $(E\, ,\rho)$ on $Y$ canonically associated to $V_*$
such that the parabolic vector bundle corresponding to $(E\, ,\rho)$ is $V_*$. 
(See \cite[Section 3]{Bi}.)

The restriction of $E$ to the complement $Y\setminus \phi^{-1}(R_X)$ is the
pullback $\phi^* (V\vert_{X\setminus R_X})$ (see \cite[p. 313, (3.3)]{Bi} and
the definition of $W$ in \cite[p. 313]{Bi}). Since $R_X\, \subset\, 
\Delta$, this implies that the nonsingular connection $D'\vert_{X\setminus 
\Delta}$ on $V\vert_{X\setminus\Delta}$ pulls back to a nonsingular connection 
$D''$ on $E\vert_{Y\setminus \phi^{-1}(\Delta)}\,=\, 
(\phi^*V)\vert_{Y\setminus \phi^{-1}(\Delta)}$.

\begin{proposition}\label{prop2}
The nonsingular connection $D''$ on $E\vert_{Y\setminus \phi^{-1}(\Delta)}$ 
extends to a logarithmic connection
$$
E\, \longrightarrow \, E\otimes K_Y\otimes {\mathcal 
O}_Y(\phi^{-1}(S_X)_{\rm red})\, .
$$
\end{proposition}

\begin{proof}
Let $\phi^*D'$ be the pulled back logarithmic connection on $\phi^* V$ whose 
singularity is contained in $\phi^{-1}(\Delta)$. The line bundle
$$
\phi^*{\mathcal O}_X(R_X)\, =\, {\mathcal O}_Y(N\cdot\phi^{-1}(R_X)_{\rm red})
$$
has a canonical logarithmic connection defined by the de Rham differential
$\beta\, \longmapsto\, d\beta$; this connection is singular over 
$\phi^{-1}(R_X)_{\rm red}$. Therefore, $D''$ produces a logarithmic connection
$$
\overline{D}''\, :\, \phi^*(V\otimes {\mathcal O}_X(R_X))\, \longrightarrow\,
\phi^*(V\otimes {\mathcal O}_X(R_X))\otimes K_Y\otimes {\mathcal 
O}_Y(\phi^{-1}(\Delta)_{\rm red})\, .
$$

The vector bundle $\phi^*(V\otimes {\mathcal O}_X(R_X))$ has a tautological
action of $\Gamma$ because it is the pullback of a vector bundle on $Y/\Gamma
\,=\, X$. The $\Gamma$--equivariant vector bundle $(E\, ,\rho)$ corresponding 
to $V_*$ is 
the intersection of certain $\Gamma$--invariant subsheaves of $\phi^*(V\otimes 
{\mathcal O}_X(R_X))$ (see \cite[p. 313, (3.3)]{Bi} and
the definition of $W$ in \cite[p. 313]{Bi}). In other words, $(E\, ,\rho)$ is of 
the form
$$
E\, =\, \bigcap_j F_j\, \subset\, \phi^*(V\otimes {\mathcal O}_X(R_X))\, ,
$$
where each $F_j$ is a $\Gamma$--invariant subsheaf of $\phi^*(V\otimes
{\mathcal O}_X(R_X))$.

Since ${\rm Res}(D'\, ,x)$ preserves each subspace $W^i_x$, it follows that 
${\rm Res}(D'\, ,x)$ preserves the quasiparabolic filtration in \eqref{f1}.
Also, ${\rm Res}(D'\, ,x)$ acts on each $W^i_x$ in \eqref{f3} as multiplication 
by a constant scalar. From these it follows that each $F_j$ is preserved by the 
logarithmic connection $\overline{D}''$ (see \cite[p. 313, (3.3)]{Bi}). Hence the 
connection $\overline{D}''$ on $E\vert_{Y\setminus \phi^{-1}(\Delta)}$ defined earlier 
induces the logarithmic connection 
$D=\overline{D}''|_E$ on $E$; the singular locus of $D$ is contained in $\phi^{-1}(\Delta)$.

To complete the proof, we need to show that $D$ is nonsingular in a neighborhood 
of $\phi^{-1}(R_X)$. For this we will use a local model of $\phi$.

Let $\mathbb D$ and $f_n$ be as in \eqref{d} and \eqref{d2} respectively.
Let $L$ be a holomorphic line bundle over $f_n(\mathbb D)$, and let
$$
D_L\, :\, L\, \longrightarrow\, L\otimes K_{\mathbb D}\otimes{\mathcal 
O}_{\mathbb D}(0)
$$
be a logarithmic connection with residue $m/n$ at $0\,\in\, f_n(\mathbb D)$, 
where $m$ is some integer. Then $f^*_nD_L$ is a logarithmic connection on $f^*_n 
L$ singular at $0$, and the residue of this connection at $0\, \in\,
\mathbb D$ is $m$. Consider
the holomorphic line bundle $(f^*_n L)\otimes {\mathcal O}_{\mathbb D}(m'\cdot
\{0\})$. The logarithmic connection $f^*_nD_L$ on $f^*_n L$, and the 
logarithmic connection on ${\mathcal O}_{\mathbb D}(m'\cdot\{0\})$ given by the 
de Rham differential, together define a logarithmic connection $\widetilde D$ 
on $(f^*_n L)\otimes {\mathcal O}_{\mathbb D}(m'\cdot \{0\})$. Since the 
residue of the de Rham connection on ${\mathcal O}_{\mathbb D}(m'\cdot \{0\})$ 
at $0$ is $-m'$, and the residue of $f^*_nD_L$ at $0$ is $m$, it follows that 
the residue of $\widetilde D$ at $0$ is $m-m'$. In particular, $\widetilde D$ 
is a nonsingular connection if $m\,=\, m'$. 

Using these facts repeatedly, it is straight-forward to deduce that the 
logarithmic connection $D$ on $E$ is nonsingular in a neighborhood of 
$\phi^{-1}(R_X)$.
\end{proof}

\section{Action of finite order line bundles}\label{se3}

\subsection{Fixed points of the action}\label{se3.1}

Let $X$ be a compact connected Riemann surface. If $E$ is a holomorphic vector
bundle of rank $r$ on $X$, and $L$ is a holomorphic line bundle on $X$, such 
that $E\otimes L$ is holomorphically isomorphic to $E$, then taking top 
exterior products we see that $(\bigwedge^r E)\otimes L^{\otimes r}$ is 
isomorphic to $\bigwedge^r E$, implying that the line bundle $L^{\otimes r}$ is 
holomorphically trivial.

Let $L$ be a nontrivial holomorphic line bundle on $X$ such that
$L^{\otimes r}$ is holomorphically trivial. Fix a holomorphic isomorphism of
$L^{\otimes r}$ with ${\mathcal O}_X$. Then there is a unique holomorphic 
connection $D_L$ on $L$ such that the above isomorphism takes the connection on 
$L^{\otimes r}$ induced by $D_L$ to the de Rham connection on ${\mathcal O}_X$ 
defined by $\beta\, \longmapsto\, d\beta$. Since any two isomorphisms between
$L^{\otimes r}$ and ${\mathcal O}_X$ differ by multiplication with a nonzero 
constant, the connection $D_L$ is independent of the choice of the
isomorphism between $L^{\otimes r}$ and ${\mathcal O}_X$.

Let $n$ be the order of $L$. So $n\, >\, 1$, and it is a divisor of $r$. Fix a
holomorphic isomorphism between $L^{\otimes n}$ and ${\mathcal O}_X$.

Let $s$ be the nonzero holomorphic section of $L^{\otimes n}$ given by the
constant function $1$ using the above isomorphism between $L^{\otimes n}$ and 
${\mathcal O}_X$. Define
\begin{equation}\label{g4}
Y_L\, :=\, \{z\in L\,\mid\, z^{\otimes n}\, \in\,
\text{image}(s)\}\, \subset\, L\, .
\end{equation}
Let
\begin{equation}\label{g5}
\varphi\, :\, Y_L\, \longrightarrow\, X
\end{equation}
be the restriction of the natural projection $L\, \longrightarrow
\,X$. The curve $Y_L$ is irreducible because the order of $L$ is $n$.

Since $\varphi$ is \'etale, if $(V\, ,D_V)$ is a logarithmic connection on 
$Y_L$ singular over $S\, \subset\, Y$, then $D_V$ induces a logarithmic 
connection on $\varphi_* V$ singular over $\varphi(S)$. This induced 
logarithmic connection on $\varphi_* V$ will be denoted by $\varphi_* D_V$.

\begin{theorem}\label{thm1}
Take a pair $(E\, ,D)$, where $E$ is a holomorphic vector bundle of rank $r$ 
on $X$ and $D$ is a logarithmic connection on $E$, such that the connection
$(E\otimes L\, ,D\otimes{\rm Id}_L+{\rm Id}_E\otimes D_L)$ is isomorphic to
$(E\, ,D)$. Then there is a vector bundle $V$ on $Y_L$ of rank $r/n$, and a 
logarithmic connection $D_V$ on $V$, such that $(\varphi_*V\, , \varphi_* D_V)$
is isomorphic to $(E\, ,D)$.
\end{theorem}

\begin{proof}
Since the holomorphic vector bundle $E\otimes L$ is holomorphically isomorphic 
to $E$, there is a holomorphic 
vector bundle $V\, \longrightarrow\, Y_L$ such that $\varphi_*V\,=\, E$ \cite[p. 
499, Lemma 2.1]{BP}. We will briefly recall the construction of $V$.

Fix a holomorphic isomorphism
\begin{equation}\label{h1}
H\, :\, E\, \longrightarrow\, E\otimes L
\end{equation}
such that the composition
$$\begin{xy} \xymatrix{
E \ar[r]^{\hskip-10pt H} & E\otimes L \ar[rr]^{H\otimes {Id}_L}&& E\otimes L^{\otimes 2}
\ar[rr]^{ \ \ H\otimes {Id}_{L^{\otimes 2}}} &&
\cdots \ar[rr]^{H\otimes {Id}_{L^{\otimes (n-1)\ \ Ê\textrm{ }}}}&&
E\otimes L^{\otimes n}}
\end{xy}$$
is the identity map with respect to the trivialization of $L^{\otimes n}$.
(Since the group $\text{Aut}(E)$ is divisible, such an isomorphism exists.)

Consider the holomorphic vector bundle $\varphi^*E$ on $Y_L$,
where $\varphi$ is the projection in \eqref{g5}. Let
\begin{equation}\label{h2}
V\, \subset\, \varphi^*E
\end{equation}
be the holomorphic subbundle such that for any $y\, \in\, Y_L$, the subspace
$V_y\,\subset\, (\varphi^*E)_y\,=\, E_{\varphi(y)}$ coincides with the 
kernel of the homomorphism
\begin{equation}\label{y}
H- \text{Id}_{E_{\varphi(y)}}\otimes y\,:\, E_{\varphi(y)}\, \longrightarrow
\,E_{\varphi(y)}\otimes L_{\varphi(y)}\,=\,(E\otimes L)_{\varphi(y)}
\end{equation}
(note that $y\, \in\, L_{\varphi(y)}$). The direct image $\varphi_*V$ is 
canonically identified with $E$ (see \cite[p. 499, Lemma 2.1]{BP}).

Now assume that the isomorphism $H$ in \eqref{h1} takes the logarithmic 
connection $D$ on $E$ to the logarithmic connection 
$D\otimes{\rm Id}_L+{\rm Id}_E\otimes D_L$ on $E\otimes L$. Such an 
automorphism exists because the group of all automorphisms of $E$ preserving
$D$ is divisible.

Consider the logarithmic connection $\varphi^*D$ on $\varphi^*E$. We will 
show that the subbundle $V$ in \eqref{h2} is preserved by $\varphi^*D$.

Let $0_X\, \subset\, L$ be the
image of the zero section of $L$. We note that the pullback of $L$ to
the total space of $L$ has a tautological section, and this section vanishes
exactly on $0_X$. Since $Y_L\, \subset\, L\setminus 0_X$, we get a tautological 
nowhere vanishing section
\begin{equation}\label{sigm}
\sigma \, \in\, \mathrm{H}^0(Y_L,\, \varphi^*L)\, .
\end{equation}
Let $\varphi^*D_L$ be the \emph{holomorphic} connection on $\varphi^*L$ obtained by 
pulling back the connection $D_L$ on $L$.

We will show that the section $\sigma$ in \eqref{sigm} is flat with 
respect to the connection $\varphi^*D_L$. For this, consider the holomorphic 
section
$$
\sigma^{\otimes n} \, \in\, \mathrm{H}^0(Y_L,\, (\varphi^*L)^{\otimes n})
\,=\, \mathrm{H}^0(Y_L,\, \varphi^*(L^{\otimes n}))\, .
$$
This section is nonzero because $\sigma$ is nonzero. On the other hand,
$L^{\otimes n}\,=\, {\mathcal O}_X$. Therefore, $\sigma^{\otimes n}$ is
the pullback of a nonzero constant section of $L^{\otimes n}\,=\, {\mathcal 
O}_X$. 
Recall that the connection $D_L$ is uniquely determined by the condition that 
the induced connection on $L^{\otimes n}$ coincides with the trivial connection
on ${\mathcal O}_X$ (given by the de Rham differential). Therefore, we conclude 
that $\sigma$ is flat for the connection $\varphi^*D_L$.

The section $\sigma$ sends any $y\,\in\, Y_L\, \subset\, L$ to $y$. Since
$\sigma$ is flat for the connection $\varphi^*D_L$, it follows
immediately that the homomorphism
$$
\theta\, :\, \varphi^*E \, \longrightarrow\, \varphi^*(E\otimes L)
$$
defined by $v\,\longmapsto\, v\otimes y$, where $v\, \in\, (\varphi^*E)_y$,
takes the logarithmic connection $\varphi^*D$ to the pulled back 
logarithmic connection $\varphi^*(D\otimes{\rm Id}_L+{\rm Id}_E\otimes D_L)$ on 
$\varphi^*(E\otimes L)$. The isomorphism 
$\varphi^*H$ also takes $\varphi^*D$ to $\varphi^*(D\otimes{\rm Id}_L+{\rm 
Id}_E\otimes D_L)$, because $H$ takes $D$ to $D\otimes{\rm 
Id}_L+{\rm Id}_E\otimes D_L$. Therefore, $\varphi^*H-\theta$ takes $\varphi^*D$ 
to $\varphi^*(D\otimes{\rm Id}_L+{\rm Id}_E\otimes D_L)$. Hence
$$
\text{kernel}(\varphi^*H-\theta)\, \subset\, \varphi^*E
$$
is preserved by $\varphi^*D$. But $V\,=\, \text{kernel}(\varphi^*H-\theta)$
(see \eqref{y}).

Consequently, $\varphi^* D$ induces a logarithmic connection on $V$;
we will denote this logarithmic connection on $V$ by $D_V$. Since
$\varphi_*V \,=\, E$, from the construction of $D_V$ it follows immediately 
that $\varphi_*D_V\, =\, D$.
\end{proof}

The converse of Theorem \ref{thm1} is also true:

\begin{proposition}
Let $(V\, , D_V)$ be a holomorphic vector bundle of rank $r/n$ on $Y_L$ 
equipped with a logarithmic connection $D_V$. Then the logarithmic connections 
$(\varphi_*V\, , \varphi_*D_V)$ and $((\varphi_*V)\otimes L\, , (\varphi_* 
D_V)\otimes \mathrm{Id}_L+\mathrm{Id}_{\varphi_*V}\otimes D_L)$ are isomorphic.
\end{proposition}

\begin{proof}
There is a natural holomorphic isomorphism of vector bundles
$E\, \longrightarrow\, (\varphi_*V)\otimes L$ \cite[p. 499, Proposition 
2.2(1)]{BP}. It is straight-forward to check that this isomorphism takes
the logarithmic connection $\varphi_*D_V$ to $(\varphi_* D_V)\otimes 
\mathrm{Id}_L+\mathrm{Id}_{\varphi_*V}\otimes D_L$.
\end{proof}

\subsection{Action on the cohomology of the moduli spaces}\label{se3.2}

Let $X$ be a compact connected Riemann surface of genus at least two. Fix a 
point 
$x_0\, \in\, X$. Fix an integer $r\, \geq\, 2$ together with an integer $d$ 
coprime 
to $r$. Fix a holomorphic line bundle $\xi$ on $X$ of degree $d$. Fix a 
logarithmic connection $D_\xi$ on $\xi$ singular exactly over $x_0$. Since
$\text{degree}(\xi)+\text{Res}(D_\xi,x_0)\,=\, 0$ \cite[p. 16, Theorem 3]{Oh}, 
it follows that $\text{Res}(D_\xi,x_0)\,=\, -d$.

Let ${\mathcal M}(r,D_\xi)$ denote the moduli space of logarithmic connections 
$(E\, ,D)$ on $X$ satisfying the following conditions:
\begin{itemize}
\item $\text{rank}(E)\,=\, r$,

\item $\bigwedge^r E\,=\, \xi$,

\item the connection $D$ is singular exactly over $x_0$,

\item $\text{Res}(D,x_0)\,=\, -\frac{d}{r}\text{Id}_{E_{x_0}}$, and

\item the logarithmic connection on $\bigwedge^r E\,=\,\xi$ induced by $D$ 
coincides with $D_\xi$.
\end{itemize}
(See \cite{Ni}, \cite{Si} for the construction of the moduli space.)

Take any $(E\, ,D)\,\in\, {\mathcal M}(r,D_\xi)$. If $F$ is a holomorphic 
subbundle of $E$ of positive rank which is preserved by $D$, then
$$
\text{degree}(F)+\text{trace}(\text{Res}(D\vert_F\, ,x_0))\,=\,
\text{degree}(F) -\frac{d}{r}\text{rank}(F)\,=\, 0
$$
\cite[p. 16, Theorem 3]{Oh}. Since $d$ is coprime to $r$, this implies that
$r\,=\, \text{rank}(F)$. Hence the logarithmic connection $(E\, ,D)$ is stable.
Therefore, the moduli space ${\mathcal M}(r,D_\xi)$ is smooth.

Let $L$ be a holomorphic line bundle over $X$ such that $L^{\otimes r}\,=\,
{\mathcal O}_X$. Let $D_L$ be the holomorphic connection on $L$ constructed in 
Section \ref{se3.1}. Let
\begin{equation}\label{pl}
\Phi\, :\, {\mathcal M}(r,D_\xi)\, \longrightarrow\, {\mathcal M}(r,D_\xi)
\end{equation}
be the automorphism defined by $(E\, ,D)\, \longmapsto\, (E\otimes L\, 
,D\otimes{\rm Id}_L+{\rm Id}_E\otimes D_L)$.

\begin{proposition}\label{prop3}
The homomorphism $\Phi^*\, :\, \mathrm{H}^*({\mathcal M}(r,D_\xi),\, {\mathbb Q})\, 
\longrightarrow\, \mathrm{H}^*({\mathcal M}(r,D_\xi),\, {\mathbb Q})$ induced by $\Phi$ 
is the identity map.
\end{proposition}

\begin{proof}
Let ${\mathcal M}_H(r,\xi)$ be the moduli space of Higgs bundles $(V\, 
,\theta)$ over $X$ such that $\text{rank}(V)\,=\, r$, $\bigwedge^r V\,=\,\xi$, 
and $\text{trace}(\theta)\,=\,0$. The moduli space ${\mathcal M}(r,D_\xi)$ is 
canonically diffeomorphic to ${\mathcal M}_H(r,\xi)$ (but this diffeomorphism 
is not holomorphic) \cite{Si2}.

The flat connection $D_L$ has finite monodromy group because the connection on
$L^{\otimes n}\,=\, {\mathcal O}_X$ induced by $D_L$ has trivial monodromy.
In particular, the connection $D_L$ is unitary. Therefore, the Higgs line bundle
corresponding to the flat line bundle $(L\, ,D_L)$ is the holomorphic line
bundle $L$ equipped with the zero Higgs field.

The holomorphic line bundle $L$ equipped with the zero Higgs field produces 
an automorphism
$$
\Phi_H\, :\, {\mathcal M}_H(r,\xi)\, \longrightarrow\, {\mathcal M}_H(r,\xi)
$$
defined by $(V\, ,\theta)\, \longmapsto\, (V\otimes L\, ,\theta\otimes 
\text{Id}_L)$. The correspondence  between flat connections and Higgs bundles is 
compatible with the operation of taking tensor products; see \cite[p. 15]{Si1}.
Consequently, the following diagram of maps
\begin{equation}\label{c}
\begin{xy}\xymatrix{
{\mathcal M}(r,D_\xi) \ar[r]^{\Phi} \ar[d]^{\wr}&{\mathcal 
M}(r,D_\xi)\ar[d]^\wr\\
{\mathcal M}_H(r,\xi) \ar[r]^{\Phi_H} &{\mathcal M}_H(r,\xi)}
\end{xy}
\end{equation}
is commutative. 

There is a universal vector bundle $\mathcal E$ over $X\times {\mathcal 
M}_H(r,\xi)$ because $d$ is coprime to $r$. The cohomology algebra 
$\mathrm{H}^*({\mathcal M}_H(r,\xi),\, {\mathbb Q})$ is generated by the K\"unneth 
components of the Chern classes $c_i({\mathcal E})$, $i\, \geq\, 0$
\cite[p. 73, Theorem 7]{Ma}, \cite[p. 641, (6.1)]{HT}. Since $L$ is of finite 
order, we have $c_1(L)\,=\, 0$. Hence
$$
c_i({\mathcal E})\,=\, c_i({\mathcal E}\otimes p^*_X L), ~\, \forall ~ i\, 
\geq\, 0\, ,
$$
where $p_X$ is the projection of $X\times {\mathcal M}_H(r,\xi)$ to $X$.
Therefore, the automorphism $$\Phi^*_H\, :\, \mathrm{H}^*({\mathcal M}_H(r,\xi),\, 
{\mathbb Q})\, \longrightarrow\,\mathrm{H}^*({\mathcal M}_H(r,\xi),\, {\mathbb Q})$$
induced by $\Phi_H$ is the identity map on the above mentioned generators of 
$\mathrm{H}^*({\mathcal M}_{H}(r,\xi),\, {\mathbb Q})$. Hence $\Phi^*_H$ is the identity 
map. Now the proposition follows from the commutativity of the diagram in 
\eqref{c}.
\end{proof}

\section{A criterion for admitting a logarithmic connection}\label{se4}

Let $X$ be a compact connected Riemann surface. Fix a finite subset
$$
\Delta\,=\, \{x_1\, ,\cdots \, ,x_n\}\, \subset\, X \, .
$$
For each point $x_i\, \in\, \Delta$, fix $\lambda_i\, \in\, {\mathbb C}$.
Let $E$ be a holomorphic vector bundle on $X$. A holomorphic vector bundle $F$ 
on $X$ of positive rank is called a \textit{direct summand} of $E$ if there is 
a holomorphic vector bundle $F'$ of positive rank such that $F\oplus F'$ is
holomorphically isomorphic to $E$. The vector bundle $E$ is called 
\textit{indecomposable} if $E$ does not have any direct summand.

If $n\,=\, 0$, the following proposition coincides with the Atiyah--Weil 
criterion for the existence of a holomorphic connection on $E$, \cite{At}, 
\cite{We}.

\begin{proposition}\label{prop4}
There is a logarithmic connection $D$ on $E$ singular over $\Delta$ with 
residue
$$
{\rm Res}(D,x_i) \,=\, \lambda_i\cdot {\rm Id}_{E_{x_i}}
$$
for every $x_i\, \in\, \Delta$ if and only if the following two conditions hold:
\begin{enumerate}
\item ${\rm degree}(E)+ {\rm rank}(E)\cdot \sum_{i=1}^n \lambda_i\,=\, 0$, and

\item ${\rm degree}(F)+ {\rm rank}(F)\cdot \sum_{i=1}^n \lambda_i\,=\, 0$
for any direct summand $F$ of $E$.
\end{enumerate}
\end{proposition}

\begin{remark}
{\rm If all $\lambda_i$ are rational numbers, then Proposition \ref{prop4} is a 
special case of the main theorem in \cite{Bi3}.}
\end{remark}

\begin{proof}
First assume that there is a logarithmic connection $D$ on $E$ singular over 
$\Delta$ with
$$
{\rm Res}(D,x_i) \,=\, \lambda_i\cdot {\rm Id}_{E_{x_i}}
$$
for all $x_i\, \in\, \Delta$. From \cite[p. 16, Theorem 3]{Oh} we conclude that
${\rm degree}(E)+ {\rm rank}(E)\cdot \sum_{i=1}^n \lambda_i\,=\, 0$.

Let $F$ be a direct summand of $E$. Fixing a complement $F'$, and a holomorphic 
isomorphism of $F\oplus F'$ with $E$, let
$$
\iota\, :\, F\, \longrightarrow\, E ~\, ~\text{ and } ~\, ~
p\, :\, E\, \longrightarrow\, F
$$
be the inclusion and projection respectively. The composition
$$\begin{xy}\xymatrix{
F \ar[r]^{\iota}& E \ar[r]^{\hskip-35pt D} & E\otimes K_X\otimes {\mathcal O}_X(\Delta)
\ar[rr]^{{p\otimes \text{Id}_{K_{\hskip-2pt X}\otimes {\mathcal 
O}_{\hskip-1pt X}(\Delta)}}} && F\otimes K_X\otimes {\mathcal O}_X(\Delta)
}\end{xy}$$
is a logarithmic connection on $F$ singular over $\Delta$ with residue
$\lambda_i\cdot {\rm Id}_{F_{x_i}}$ for each $x_i\, \in\, \Delta$.
Therefore, from \cite[p. 16, Theorem 3]{Oh} we conclude that
${\rm degree}(F)+ {\rm rank}(F)\cdot \sum_{i=1}^n \lambda_i\,=\, 0$.

To prove the converse, we will construct a holomorphic vector bundle ${\mathcal 
A}(E)$ on $X$ associated to $E$ using $\{\lambda_i\}_{i=1}^n$.

For any analytic open subset $U\, \subset\, X$, consider all pairs
of the form $(h_U\, ,D_U)$, where $h_U$ is a holomorphic function on $U$, and
$$
D_U\, :\, E\vert_U\, \longrightarrow\, (E\otimes K_X\otimes {\mathcal 
O}_X(\Delta))\vert_U
$$
is a holomorphic differential operator satisfying the 
identity
$$
D_U(f\cdot s)\, =\, f\cdot D_U(s) + h_U\cdot d(f)\otimes s\, ,
$$
where $f$ is any locally defined holomorphic function on $X$ and $s$ is any 
locally defined holomorphic section of $E$. The above identity implies that the 
order of $D_U$ is at most one; the order of $D_U$ is one 
if and only if $h_U$ is not identically zero. By the Poincar\'e adjunction 
formula, the fiber of the line bundle $\mathcal{O}_X(\Delta)$ over any $x_i\, 
\in\, \Delta$ is identified to the tangent space $T_{x_i}X$. Therefore, the fiber
of $K_X\otimes {\mathcal O}_X(\Delta)$ over any $x_i\, \in\, 
\Delta$ is canonically identified with $\mathbb C$. For any 
$(h_U\, ,D_U)$ as above, and any $x_i\, \in\, \Delta\cap U$, the composition
\begin{equation}\label{zh3}
E\vert_U\, \stackrel{D_U}{\longrightarrow}\, (E\otimes K_X\otimes {\mathcal
O}_X(\Delta))\vert_U\, \stackrel{\text{ev}_{x_i}}{\longrightarrow}\,E_{x_i}\, ,
\end{equation}
where $\text{ev}_{x_i}$ is the evaluation at $x_i$, is ${\mathcal 
O}_U$--linear. Let ${\mathcal A}(E)\vert_U$ be the space of all pairs
$(h_U\, ,D_U)$ of the above type such that the composition in \eqref{zh3}
is multiplication by $h_U(x_i)\cdot \lambda_i$ for all $x_i\, \in\,
\Delta\cap U$. This ${\mathcal A}(E)\vert_U$ is an ${\mathcal O}_U$--module. 
Indeed,
$$
f\cdot (h_U\, ,D_U)\,=\, (f\cdot h_U\, ,f\cdot D_U)\, .
$$
It is straight-forward to check that the map $U\, 
\longmapsto\, {\mathcal A}(E)\vert_U$ defines a torsionfree coherent analytic 
sheaf on $X$. The holomorphic vector bundle on $X$ defined by this
torsionfree coherent analytic sheaf will be denoted by ${\mathcal A}(E)$.
Take any $(h_U\, ,D_U)\, \in\, {\mathcal A}(E)\vert_U$. If $h_U\,=\,0$,
then $D_U$ is evidently a holomorphic section of $\mathrm{End}(E)\otimes K_X$ over $U$.
Hence $\mathrm{End}(E)\otimes K_X$ is a subbundle of ${\mathcal A}(E)$. Therefore,
we have a short exact sequence of holomorphic vector bundles on $X$
\begin{equation}\label{zh2}
0\, \longrightarrow\, \mathrm{End}(E)\otimes K_X\, \longrightarrow\, {\mathcal A}(E)
\, \stackrel{\nu}{\longrightarrow}\, {\mathcal O}_X \, \longrightarrow\, 0\, ,
\end{equation}
where $\nu$ sends any $(h_U\, ,D_U)$ to $h_U$.

It now suffices to prove that this exact sequence splits. Indeed, 
it is straight-forward to see that a logarithmic connection $D$ on $E$ 
singular over $\Delta$ with residue
$$
{\rm Res}(D,x_i) \,=\, \lambda_i\cdot {\rm Id}_{E_{x_i}}
$$
for all $x_i\, \in\, \Delta$ is a homomorphism of holomorphic vector bundles
$$
D\, :\, {\mathcal O}_X\, \longrightarrow\,{\mathcal A}(E)
$$
such that $\nu\circ D\,=\, \text{Id}_{{\mathcal O}_X}$, where $\nu$ is the
homomorphism in \eqref{zh2}.

The vector bundle $E$ is a direct sum of indecomposable vector bundles. 
Therefore, it is enough to prove the converse under the
assumption that $E$ is indecomposable. So assume that $E$ is indecomposable.

The obstruction to a holomorphic splitting of the exact sequence in
\eqref{zh2} is a cohomology class
\begin{equation}\label{t}
\theta\, \in\,\mathrm{H}^1(X,\, \mathrm{End}(E)\otimes K_X)\,=\, 
\mathrm{H}^0(X,\, \mathrm{End}(E))^*
\end{equation}
(Serre duality).

Since $E$ is indecomposable, all holomorphic endomorphisms of $E$ are
of the form $c\cdot \text{Id}_E +N$, where $c\, \in\, \mathbb C$ and
$N$ is a nilpotent endomorphism \cite[p. 201, Proposition 16]{At}. It
can be shown that
\begin{equation}\label{id}
\theta(\text{Id}_E)\,=\,0
\end{equation}
(see \cite[p. 202, Proposition 18(i)]{At}). To prove \eqref{id}, consider the 
line bundle ${\mathcal L}\, :=\, \bigwedge^{{\rm rank}(E)} E$. Let
${\mathcal A}({\mathcal L})$ be the vector bundle on $X$ constructed exactly
as ${\mathcal A}(E)$ by replacing $E$ by $\mathcal L$ and replacing
each $\lambda_i$ by ${\rm rank}(E)\cdot \lambda_i$. Let
$$
0\, \longrightarrow\, K_X\, \longrightarrow\, {\mathcal A}({\mathcal L})
\, \longrightarrow\, {\mathcal O}_X \, \longrightarrow\, 0
$$
be the exact sequence constructed as in \eqref{zh2}, and let
$$
\theta_0\, \in\, \mathrm{H}^1(X,\, K_X)\,=\, \mathrm{H}^0(X,\,{\mathcal O}_X)^* \,=\, \mathbb C
$$
be the obstruction to its splitting (as defined in \eqref{t}). Then,
\begin{equation}\label{id2}
\theta(\text{Id}_E)\,=\, \theta_0(1)\, .
\end{equation}
Since ${\rm degree}(E)+ {\rm rank}(E)\cdot \sum_{i=1}^n \lambda_i\,=\, 0$, the
line bundle ${\mathcal L}$ has a logarithmic connection $D_0$ singular over
$\Delta$ with residue
$$
{\rm Res}(D_0,x_i) \,=\, {\rm rank}(E)\cdot \lambda_i
$$
for every $x_i\, \in\, \Delta$. But this implies that $\theta_0\,=\, 0$.
Hence \eqref{id} follows from \eqref{id2}.

The proof of Proposition 18(ii) in \cite{At} yields that
$\theta(N)\,=\, 0$ for any nilpotent endomorphism of $E$. Hence
$\theta\,=\, 0$. Therefore, the short exact sequence in \eqref{zh2}
splits holomorphically. Consequently, $E$ admits a logarithmic connection $D$ 
singular over $\Delta$ such that ${\rm Res}(D,x_i) \,=\, \lambda_i\cdot {\rm 
Id}_{E_{x_i}}$ for all $x_i\, \in\, \Delta$. 
\end{proof}


\end{document}